\documentclass[11pt]{amsart}
\usepackage{amsthm, amsfonts, amssymb, amsmath, enumitem, bbold}
\usepackage[all]{xy}
\usepackage{avant}

\newcommand{\C}{\mathbb{C}}
\newcommand{\R}{\mathbb{R}}

\newcommand{\N}{\mathbb{N}}

\newcommand{\G}{\Gamma}

\newcommand{\e}{\varepsilon}

\newcommand{\la}{\langle}
\newcommand{\ra}{\rangle}
\newcommand{\cleq}{\preccurlyeq}
\newcommand{\ct}{\mathbb{1}}

\newtheorem{thm}{Theorem}[section]
\newtheorem{lem}[thm]{Lemma}

\theoremstyle{definition}
\newtheorem{defn}[thm]{Definition}

\newtheorem{rem}[thm]{Remark}

\newtheorem*{ack}{Acknowledgments}

\begin{document}
\title[Norms of spherical averages]{Norms of spherical averaging operators\\ for some geometric group actions} 
\author{Bogdan Nica}
\date{\today}
\subjclass{Primary: 43A15, 20F67. Secondary: 22D15, 22D20, 20F65, 20F69.}
\address{\newline Department of Mathematical Sciences\newline Indiana University Indianapolis}
\email{bnica@iu.edu}
\begin{abstract} We obtain asymptotic estimates for the $\ell^p$-operator norm of spherical averaging operators associated to certain geometric group actions. The motivating example is the case of Gromov hyperbolic groups, for which we obtain asymptotically sharp estimates. We deduce asymptotic lower bounds for the combinatorial expansion of spheres. 
\end{abstract}
\maketitle

\section{Averaging operators}
Let $\G$ be a countable discrete group, and let $S\subseteq \G$ be a non-empty finite subset. The \emph{averaging operator} $\lambda_S$ is given by
\begin{align*}
\lambda_S=\frac{1}{|S|}\sum_{g\in S} \lambda(g)
\end{align*}
where $\lambda$ is the left regular representation of $\G$. More explicitly, $\lambda_S$ acts on complex-valued functions defined on $\G$, sending a function $\phi: \G\to \C$ to the function $\lambda_S(\phi): \G\to \C$ given by 
\[\lambda_S(\phi)(h)=\frac{1}{|S|}\sum_{g\in S} \phi(g^{-1}h).\]

The averaging operator $\lambda_S$ is a natural and fundamental operator, which appears under several guises. In the terminology of random walks, it is the transition operator for the simple random walk defined by $S$. In graph theory, $\lambda_S$ is closely related to the discrete Laplacian. Averaging operators also appear in harmonic analysis, and in ergodic theory.

Let $p\in [1,\infty)$. The averaging operator $\lambda_S$ is a bounded operator on $\ell^p\G$, and the broad problem we are interested in is that of \emph{computing or estimating the $p$-operator norm $\|\lambda_S\|_{p\to p}$}.

There are four general facts that frame this broad problem. 

Firstly,
\begin{align}\label{eq:triv-upper}
\|\lambda_S\|_{p\to p}\leq 1,
\end{align}
with equality for $p=1$. We will therefore restrict to $p\in (1,\infty)$ in what follows. More interestingly, equality in \eqref{eq:triv-upper} also holds when $\G$ is an amenable group. Kesten \cite{Kes2} originally proved this fact in  the case $p=2$; the same argument, a nice application of F\o lner sets, works in fact for each $p\in (1,\infty)$. The upshot is that our problem is interesting only when $\G$ is non-amenable.

Secondly, we have the trivial lower bound 
\begin{align}\label{eq:triv-lower}
\|\lambda_S\|_{p\to p} \geq |S|^{-1/p'}
\end{align}
where $p'$ is the conjugate exponent of $p\in (1,\infty)$, given by $1/p+1/p'=1$. Indeed, $\lambda_S$ maps $\ct_e$, the characteristic function of the identity of $\G$, to $|S|^{-1}\cdot \ct _S$, the normalized characteristic function of the subset $S$. It follows that $\|\lambda_S\|_{p\to p} \geq \||S|^{-1}\cdot \ct _S\|_p=|S|^{-1+1/p}=|S|^{-1/p'}$, as claimed.

Thirdly, we have the duality formula
\begin{align}\label{eq:dual}
\|\lambda_S\|_{p\to p}=\|\lambda_{S^{-1}}\|_{p'\to p'}.
\end{align}
This can be deduced from the identity $\la \lambda_S (\phi), \psi\ra=\la \phi, \lambda_{S^{-1}}(\psi)\ra$, for $\phi\in \ell^p\G$ and $\psi\in\ell^{p'}\G$, where $\la \cdot, \cdot \ra: \ell^p\G\times \ell^{p'}\G\to \C$ is the duality pairing given by $\la \phi, \psi\ra=\sum_{h\in \G} \phi(h)\:\overline{\psi(h)}$. The duality formula \eqref{eq:dual} allows a back-and-forth between the two ranges of exponents, $p\in (1,2]$ and $p\in [2,\infty)$. From now on, we focus on the range $p\in (1,2]$. Our preference for the lower range stems from a circle of ideas, results and open questions addressing the Banach algebra completions of the group algebra $\C\G$ under the left regular representation $\lambda$, for varying exponents $p$. The classical completions occur at $p=1$, yielding the Banach algebra $\ell^1\G$, respectively at $p=2$, yielding the reduced $C^*$-algebra $C^*_r\G$. The group Banach algebras indexed by $p\in (1,2)$ interpolate between the `easy' $\ell^1\G$ and the `hard' $C^*_r\G$. First introduced by Carl Herz in the early 1970's, this interpolating family is currently enjoying some renewed interest, see Liao--Yu \cite{LY}, Phillips \cite{Phi}, Samei--Wiersma \cite{SW, SW2}, Gardella--Thiel \cite{GT}. 

Fourthly, for $p\in (1,2)$ we have the interpolation bound
\begin{align}\label{eq: interpol}
\|\lambda_S\|_{p\to p}\leq \|\lambda_S\|^{2/p'}_{2\to 2}.
\end{align}
Indeed, Riesz-Thorin interpolation yields $\|\lambda_S\|_{p\to p}\leq \|\lambda_S\|^{1-\theta}_{1\to 1}\: \|\lambda_S\|^{\theta}_{2\to 2}$, where $\theta$ is given by $1/p=(1-\theta)+\theta/2$. Thus $\theta=2/p'$, and we recall that $\|\lambda_S\|_{1\to 1}=1$. The interpolation bound \eqref{eq: interpol} is promising--it implies that an upper bound for the $2$-operator norm $\|\lambda_S\|_{2\to 2}$ leads to upper bounds for $p$-operator norm $\|\lambda_S\|_{p\to p}$ in the range $p\in (1,2)$. It turns out, however, that the interpolation approach may not provide optimal upper bounds for $p$-operator norms; this is one takeaway from the results described in the next section.

We close this introductory section with the remark that, in certain geometric situations, it may be more natural to consider the right averaging operator $\rho_S$; this acts by sending a function $\phi: \G\to \C$ to the function $\rho_S(\phi): \G\to \C$ given by 
\[\rho_S(\phi)(h)=\frac{1}{|S|}\sum_{g\in S} \phi(hg).\]
The left averaging operator $\lambda_S$ and the right averaging operator $\rho_S$ are conjugate on $\ell^p\G$ via the isometric isomorphism $J: \ell^p\G\to \ell^p\G$, $J\phi(h)=\phi(h^{-1})$. In particular, $\lambda_S$ and $\rho_S$ have the same $p$-operator norm.

\section{Spherical averaging operators on hyperbolic groups} 
Nearly all results on $p$-operator norms of averaging operators that can be found in the literature address the case when $p=2$. The $2$-operator norm of $\lambda_S$, $\|\lambda_S\|_{2\to 2}$, is often referred to as a spectral radius for $S$. A classical direction is that of taking $S$ to be a symmetric generating set of $\G$. The seminal result of Kesten \cite{Kes} settles the case when $\G$ is a free group and $S$ is the standard symmetric generating set. Kesten-type formulas or estimates have been pursued for a number of other finitely generated groups, and natural symmetric generating sets. A very interesting study case is that of surface groups, see \cite{BC, Nag, Z, B, G}.

It is natural to consider the averaging operator $\lambda_S$ over `geometric' subsets $S$. Our main focus is on the case when $S$ is a `sphere' in $\G$. Let $l:\G\to [0,\infty)$ be a proper length function; for instance, $l$ could be the word-length defined by a finite symmetric generating subset of $\G$. The \emph{sphere} of radius $n$ is the (symmetric) subset of $\G$ given by
\begin{align*}
S(n)=\{g\in \G: l(g)=n\}. 
\end{align*}
Similarly, the \emph{ball} of radius $n$ is the (symmetric) subset of $\G$ given by $B(n)=\{g\in \G: l(g)\leq n\}$. Averaging operators over balls are also considered in this paper, though they only play a supporting role.

In the case when $\G$ is a free group endowed with the standard word-length, the $p$-operator norm of $\lambda_{S(n)}$ has been explicitly computed by Cohen \cite{Coh} for $p=2$, respectively by Pytlik \cite{Pyt81, Pyt84} for arbitrary $p\in (1,\infty)$. 

\begin{thm}[Cohen--Pytlik]
Let $\G$ be a free group on $k\geq 2$ generators, endowed with the standard word-length. Then
\begin{align}\label{eq: coh}
\big\|\lambda_{S(n)}\big\|_{2\to 2}= \big((1-1/k)n+1\big) \cdot (2k-1)^{-n/2},
\end{align}
and, for $p\in (1,2)$,
\begin{align}\label{eq: pyt}
\big\|\lambda_{S(n)}\big\|_{p\to p}=C(1/p)\cdot (2k-1)^{-n/p'}+C(1/p')\cdot (2k-1)^{-n/p}
\end{align}
where $C(t)$ is an explicit rational expression in terms of $k$ and $t$.
\end{thm}

Kesten's formula, mentioned above, is the case $n=1$ of Cohen's formula \eqref{eq: coh}. Pytlik's approach in \cite{Pyt81} also recovers Cohen's formula \eqref{eq: coh}. But the key outcome of Pytlik's approach is formula \eqref{eq: pyt}; this is the only result known to us that addresses the $p$-operator norm of an averaging operator for $p\neq 2$.

The Cohen--Pytlik formulas are undoubtedly remarkable. It has to be acknowledged, however, that such exact computations are extremely rare, and they owe to very special circumstances--free group, standard word-length. (There is one more exact computation we know of, due to Cartwright and M\l{}otkowski \cite{CM}; it addresses groups acting on triangle buildings, for $p=2$.) It is not at all clear how to extend the Cohen--Pytlik formulas  to other groups. The fragility of the formulas \eqref{eq: coh} and \eqref{eq: pyt} is apparent, even over a free group, as soon as we consider changing the word-length, or replacing the sphere $S(n)$ by the ball $B(n)$.

Our contention is that the Cohen--Pytlik exact formulas become more meaningful when viewed in a simplified, asymptotic form. Let us explain what asymptotic equivalence means, for we will use it throughout the paper. Given two functions $f_1,f_2:\N\to (0,\infty)$, we write $f_1\asymp f_2$ if there are positive constants $c,C$ so that $c f_1(n)\leq f_2(n)\leq C f_1(n)$ for all $n\in \N$. The asymptotic viewpoint on the Cohen--Pytlik exact formulas is in agreement with the general philosophy of geometric group theory; after all, word-lengths on finitely generated groups are asymptotically equivalent in an analogous way. 

\begin{thm}[Cohen--Pytlik, asymptotic form]\label{CP} Let $\G$ be a non-abelian free group, endowed with the standard word-length. Then
\begin{align}\label{eq: cohen}
\big\|\lambda_{S(n)}\big\|_{2\to 2}&\asymp (n+1)\:|S(n)|^{-1/2},
\end{align}
and, for $p\in (1,2)$,
\begin{align}\label{eq: pytlik}
\big\|\lambda_{S(n)}\big\|_{p\to p}&\asymp|S(n)|^{-1/p'}.
\end{align}
\end{thm}

To clarify, formulas \eqref{eq: cohen} and \eqref{eq: pytlik} describe the behavior of the $p$-operator norms $\|\lambda_{S(n)}\|_{p\to p}$ as $n$ increases; the implied multiplicative constants depend on the rank of $\G$ and on $p$, but not on $n$.

We highlight three aspects of the formulas \eqref{eq: cohen} and \eqref{eq: pytlik}. Firstly, they exhibit an interesting discrepancy: the case $p=2$ has an additional radial factor. We do not have a conceptual explanation for this discontinuity at $p=2$. Secondly, they witness that interpolation may not yield optimal bounds, for the estimate \eqref{eq: pytlik} is better than what \eqref{eq: interpol} and \eqref{eq: cohen} would predict. Thirdly, we see that, for $p\in (1,2)$, the estimate \eqref{eq: pytlik} asymptotically matches the trivial lower bound \eqref{eq:triv-lower}. 

In turns out that the Cohen--Pytlik asymptotic formulas \eqref{eq: cohen} and \eqref{eq: pytlik} hold, much more generally, for Gromov hyperbolic groups equipped with any word-length function. The case $p=2$ is an instance of a more general result from \cite{Nic17}.

 \begin{thm}[\cite{Nic17}]\label{thm: haag} Let $\G$ be a non-elementary hyperbolic group, endowed with a word-length. Then:
\begin{align*}
\big\|\lambda_{S(n)}\big\|_{2\to 2}&\asymp (n+1)\:|S(n)|^{-1/2}.
\end{align*}
\end{thm}

In this paper we handle the case $p\in (1,2)$.

 \begin{thm} \label{thm: main}
 Let $\G$ be a non-elementary hyperbolic group, endowed with a word-length, and let $p\in (1,2)$. Then:
\begin{align*}
\big\|\lambda_{S(n)}\big\|_{p\to p}&\asymp |S(n)|^{-1/p'}.
\end{align*}
\end{thm}

The strategy towards Theorem~\ref{thm: main} is very different from the one used in \cite{Nic17} in order to obtain Theorem~\ref{thm: haag}. In the case $p=2$, the upper bound is granted by property RD, and the issue is to obtain a matching lower bound; this is achieved by using the boundary of $\G$. For the range $p\in (1,2)$, it is the lower bound that is known--namely, it is the trivial lower bound \eqref{eq:triv-lower}--and we are aiming for a matching upper bound. It is unclear whether the boundary of $\G$ could be used to this end. Instead, we adapt Pytlik's approach to \eqref{eq: pyt}. Although the proof of Theorem~\ref{thm: main} is ultimately carried out in a different way, our use of the Busemann cocyle (a group cocycle on $\G$) is inspired by Pytlik's use of the Poisson kernel (a group cocycle on the boundary of $\G$). By avoiding the boundary, we are actually able to formulate a much more general result--see Theorem~\ref{thm: spheres}.

Estimates for $p$-operator norms of an averaging operator $\lambda_S$ have an interesting by-product: estimates for the combinatorial expansion of the subset $S$. Here is an informal description of what this means; we refer to Section~\ref{sec: exp} for the precise definition. Consider a product set $SX=\{sx: s\in S, x\in X\}\subseteq \G$, where $X$ is an arbitrary finite subset of $\G$; trivially, we have the upper bound $|SX|\leq |S||X|$. The combinatorial expansion of $S$ encodes lower bounds for $|SX|$ relative to $|X|$, uniformly in $X$.

Informally, the next result says that the sequence of spheres in a hyperbolic group forms an asymptotic expander.
 
\begin{thm}\label{thm: exphyp}
Let $\G$ be a non-elementary hyperbolic group, endowed with a word-length. Then there exists a constant $c\in (0,1)$ such that, for each $n$, we have
\[|S(n)X|\geq c |S(n)||X|\]
for any finite subset $X\subseteq \G$. 
\end{thm}

The proofs of Theorems~\ref{thm: main} and ~\ref{thm: exphyp} are completed in Section~\ref{sec: hyp}.


\section{The cocycle bound}
Let $\G$ be a discrete countable group, and $p\in (1,\infty)$. The left regular representation $\lambda$ of $\G$ on $\ell^p\G$ is the isometric representation given by $\lambda(g)\phi=g.\phi$, where $g.\phi(h)=\phi(g^{-1}h)$. We extend $\lambda$ by linearity to the group algebra $\C\G$, setting
\begin{align*}
\lambda(a)=\sum_{g\in \G} a(g)\lambda(g)
\end{align*}
for any finitely supported function $a:\G\to \C$. 

The main result of this section is an upper bound on the $p$-operator norm $\|\lambda(a)\|_{p\to p}$, in terms of an additional ingredient.

Let $\kappa: \G\to \ell^\infty\G$ be a cocycle. This means that the identity
\begin{align}\label{eq: cocycle}
\kappa(g_1g_2)=\kappa(g_1)+g_1. \kappa(g_2)
\end{align}
holds for all $g_1,g_2\in \G$. The cocycle identity implies that $\kappa(1)=0$, and that
\begin{align}\label{eq: switch}
\kappa(g^{-1})=-g^{-1}.\kappa(g)
\end{align}
for all $g\in \G$. 

Let the \emph{$\kappa$-norm} of a finitely supported function $a:\G\to \C$ be given as follows:
\begin{align}
N_\kappa(a)=\sup_{h\in \G}\: \sum_{g\in \G} \big|a(g)\big| \:e^{\kappa(g)(h)}= \bigg\|\sum_{g\in \G} \big|a(g)\big| \:e^{\kappa(g)}\bigg\|_\infty.
\end{align}

Evidently, $N_\kappa$ is a ($\C$-vector space) norm on $\C\G$; this holds for any map $\kappa: \G\to \ell^\infty\G$. When $\kappa$ is a cocycle, it can be checked that $N_\kappa$ is also submultiplicative. Therefore the $\kappa$-norm $N_\kappa$ is an algebra norm on $\C\G$. 

For the trivial cocycle $\kappa=0$ we have $N_\kappa(a)=\|a\|_1$. For a general cocycle $\kappa$, we might think of the $\kappa$-norm $N_\kappa$ as a twisted $\ell^1$-norm on $\C\G$. 

\begin{lem}\label{lem: cocycle} Let $p\in (1,\infty)$, with conjugate exponent $p'$. For any finitely supported function $a:\G\to \C$, we have
\begin{align}\label{eq: cocyclebound}
\|\lambda(a)\|_{p\to p} \leq N_{p\kappa}(a^*)^{1/p}\: N_{p'\kappa}(a)^{1/p'}.
\end{align}
\end{lem}

Here $a^*:\G\to \C$ is the function given by $a^*(g)=\overline{a(g^{-1})}$. For the trivial cocycle $\kappa=0$, the bound \eqref{eq: cocyclebound} is simply the trivial upper bound $\|\lambda(a)\|_{p\to p} \leq \|a\|_1$. Thus \eqref{eq: cocyclebound} could be thought of as a twisting of the trivial upper bound.

\begin{proof} Let $\phi\in \ell^p\G$. Then
\begin{align*}
\lambda(a)\phi(h)=\sum_{g\in \G} a(g)\phi(g^{-1}h).
\end{align*}
Using H\"older's inequality, we have:
\begin{align*}
\big|\lambda(a)\phi(h)\big|&\leq \sum_{g\in \G} \big|a(g)\big|\big|\phi(g^{-1}h)\big|\\
&=\sum_{g\in \G} \big|a(g)\big|^{1/p'} e^{\kappa(g)(h)}\:\big|a(g)\big|^{1/p}\big|\phi(g^{-1}h)\big| \: e^{-\kappa(g)(h)}\\
&\leq \Big(\sum_{g\in \G} \big|a(g)\big|\: e^{p' \kappa(g)(h)}\Big)^{1/p'} \Big(\sum_{g\in \G}  \big|a(g)\big|\big|\phi(g^{-1}h)\big|^pe^{-p \kappa(g)(h)}\Big)^{1/p}\\
&\leq N_{p'\kappa}(a)^{1/p'}\:\Big(\sum_{g\in \G} \big|a(g)\big|\big|\phi(g^{-1}h)\big|^p e^{-p \kappa(g)(h)} \Big)^{1/p}.
\end{align*}
We deduce that
\begin{align*}
\|\lambda(a)\phi\|_p&=\Big(\sum_{h\in \G} \big|\lambda(a)\phi(h)\big|^p\Big)^{1/p}\\
&\leq N_{p'\kappa}(a)^{1/p'}\: \Big(\sum_{h\in \G} \sum_{g\in \G} \big|a(g)\big| \big|\phi(g^{-1}h)\big|^pe^{-p \kappa(g)(h)} \Big)^{1/p}.
\end{align*}
Next, we estimate the latter double sum. We write
\begin{align*}
\sum_{h\in \G} \sum_{g\in \G} \big|a(g)\big| \big|\phi(g^{-1}h)\big|^pe^{-p \kappa(g)(h)}&=\sum_{g\in \G}\big|a(g)\big|\sum_{h\in \G}  \big|\phi(g^{-1}h)\big|^p e^{-p \kappa(g)(h)}\\
&=\sum_{g\in \G}\big|a(g)\big|\sum_{h\in \G} \big|\phi(h)\big|^p e^{-p \kappa(g)(gh)}\\
&=\sum_{h\in \G} \Big( \sum_{g\in \G}\big|a(g)\big|\: e^{-p \kappa(g)(gh)}\Big)\: \big|\phi(h)\big|^p.
\end{align*}

In the second step of the above computation, we have made the change of variable $h:=gh$ for each $g\in \G$. Now, by \eqref{eq: switch}, we have $-\kappa(g)(gh)=\kappa(g^{-1})(h)$ and so
\begin{align*}
 \sum_{g\in \G}\big|a(g)\big|\: e^{-p \kappa(g)(gh)}&= \sum_{g\in \G}\big|a(g)\big|\: e^{p \kappa(g^{-1})(h)}\\
 &=\sum_{g\in \G} \big|a^*(g)\big| \:e^{p\kappa(g)(h)}\leq N_{p\kappa}(a^*).
\end{align*}
We have made another change of variable along the way, namely $g:=g^{-1}$. We infer that
\begin{align*}
\sum_{h\in \G} \sum_{g\in \G} \big|a(g)\big| \big|\phi(g^{-1}h)\big|^pe^{-p \kappa(g)(h)} \leq N_{p\kappa}(a^*) \sum_{h\in \G} \big|\phi(h)\big|^p
\end{align*}
and, consequently, that
\begin{align*}
\|\lambda(a)\phi\|_p\leq N_{p'\kappa}(a)^{1/p'}\:N_{p\kappa}(a^*)^{1/p}\:  \|\phi\|_p.
\end{align*}
The desired bound \eqref{eq: cocyclebound} follows.
\end{proof}

The above proof only used the relation \eqref{eq: switch}, not the full power of the cocycle identity \eqref{eq: cocycle}. 
We find it natural to work with cocycles, however, since they are the main source for the relation \eqref{eq: switch}. Additionally, as we have noted, the $\kappa$-norm is an algebra norm on the group algebra $\C\G$ whenever $\kappa$ is a cocycle.

In keeping with our asymptotic viewpoint, we should consider the stability of the cocycle bound \eqref{eq: cocyclebound}. Here are two remarks. Firstly, if $\tilde \kappa$ is a cocycle which is a uniformly bounded perturbation of $\kappa$, meaning that $\sup_{g\in \G}\|\tilde \kappa(g)-\kappa(g)\|_\infty<\infty$, then the upper bound in \eqref{eq: cocyclebound} for the two cocycles, $\tilde \kappa$ and $\kappa$, is asymptotically the same. Secondly, the cocycle bound \eqref{eq: cocyclebound} can be extended without difficulty to quasi-cocycles. Recall, this means that $\kappa: \G\to \ell^\infty\G$ satisfies the cocycle identity \eqref{eq: cocycle} up to a uniformly bounded error: $\sup_{g,h\in \G}\|\kappa(gh)-\kappa(g)-g. \kappa(h)\|_\infty<\infty$.


\section{A geometric application of the cocycle bound}
As before, let $\G$ be a discrete countable group, and $p\in (1,\infty)$. We pursue the following instance of the cocycle setup, discussed in the previous section: the finitely supported function $a:\G\to \C$ is $\ct _S$, the characteristic function of the subset $S$, while the cocycle is a classical one, of geometric origin. 

Let $X$ be a metric space on which $\G$ acts properly by isometries. The distance between two points $x$ and $y$ in $X$ is denoted by $d(x,y)$. The \emph{Busemann cocycle} $\beta: \G\to \ell^\infty\G$ with respect to a basepoint $o\in X$ is defined as follows: for each $g\in \G$, $\beta(g):\G\to \C$ is the map
\[h\mapsto d(o,ho)-d(go,ho).\]
We have, indeed, that $\beta(g)\in \ell^\infty\G$; in fact, $\|\beta(g)\|_\infty=d(o,go)$. The cocycle identity, $\beta(g_1g_2)=\beta(g_1)+g_1.\beta(g_2)$ for all $g_1,g_2\in \G$, is easily checked. Note also that a change of the basepoint $o$ entails a uniformly bounded perturbation of the Busemann cocycle $\beta$. So, for the purposes of the cocycle bound \eqref{eq: cocyclebound}, the dependence of the Busemann cocycle on the basepoint is harmless.

We will actually work with scalings of the Busemann cocycle, that is, cocycles of the form $\e \beta$ where $\e$ is some positive parameter. This flexibility will allow us to optimize the cocycle bound \eqref{eq: cocyclebound}. Recall that we take $a=\ct_S$, so $a^*=\ct_{S^{-1}}$. For these choices, Lemma~\ref{lem: cocycle} implies the following.

\begin{lem}\label{lem: g-cocycle} 
Let $\beta: \G\to \ell^\infty\G$ be the Busemann cocycle with respect to a basepoint $o\in X$, and let $\e>0$. 
Then for any finite subset $S\subseteq \G$ we have
\begin{align}\label{eq: concrete}
\big\|\lambda(\ct_S)\big\|_{p\to p}\leq  \bigg\|\sum_{g\in S^{-1}} e^{p\e \beta(g)}\bigg\|_\infty^{1/p}\:  \bigg\|\sum_{g\in S} e^{p'\e \beta(g)}\bigg\|_\infty^{1/p'}.
\end{align}
\end{lem}

Our next concern is, of course, that of upper bounding the norms appearing on the right-hand side of \eqref{eq: concrete}. We are mainly interested in $S$ being a sphere or a ball in $\G$, with respect to the length function $l(g)=d(o,go)$ induced by some fixed basepoint $o\in X$. That is to say, we are interested in $S$ being one of the following:
\[S_o(n)=\{g\in \G: d(o,go)=n\}, \qquad B_o(n)=\{g\in \G: d(o,go)\leq n\}.\]

We now introduce some key geometric hypotheses on the metric space $X$. We start with a notation. Given two points $x,y\in X$ and $\rho\geq 0$, we consider the rough geodesic segment
\begin{align*}
[x,y]_\rho=\big\{z\in X: d(x,z)+d(z,y)\leq d(x,y)+\rho\big\}.
\end{align*}

\begin{defn}
A metric space $X$ is \emph{roughly modular} if there is some constant $\rho\geq 0$ so that the following holds: for any three points $x,y,z\in X$, there exists a point 
\begin{align*}
m\in [x,y]_\rho\cap [y,z]_\rho \cap [z,x]_\rho.
\end{align*}
Such a point $m$ is said to be a \emph{rough median} of $x,y,z$.
\end{defn}

Modularity, corresponding to the case $\rho=0$, is an already established terminology in graph theory. The adjective `rough' is the functor which relaxes metric equalities by allowing an additive bounded error. 

For the remainder of this section, the following \textbf{standing assumptions} on $X$ are in place:
\begin{itemize}[leftmargin=45pt]
\item[\textsc{mod:} ]$X$ is a roughly modular graph, with rough modularity constant $\rho$;
\item[\textsc{pol($d$):} ] there exists a non-negative integer $d$ and $C'>0$ such that, for each $x,y\in X$ and $n\in \N$ we have 
\[\big| \{z\in [x,y]_\rho: d(x,z)=n\}\big|\leq C' (n+1)^d.\]
\end{itemize}
Coincidentally, these geometric assumptions also feature in Lafforgue's proof that cocompact lattices in $\mathrm{SL}_3(\R)$ and $\mathrm{SL}_3(\C)$ satisfy property RD \cite[Defn.2.2]{Laf}; cf. the Chatterji--Ruane criterion \cite[Prop.1.7]{CR}.

\begin{lem}\label{lem: work}
Assume that the isometric action of $\G$ on $X$ satisfies the condition
\begin{itemize}[leftmargin=45pt]
\item[\textsc{exp($\delta$):}] there exist $\delta>0$ and $C>0$ such that, for each $x,y\in X$ and $n\in \N$, we have
$\big|\{g\in \G: d(x,gy)\leq n\}\big|\leq Ce^{\delta n}$.
\end{itemize}

Let $\beta$ be the Busemann cocycle with respect to a basepoint $o\in X$, and let $\e>0$. Then for any subset $S\subseteq B_o(n)$ we have:
\begin{align}\label{eq: cases}
\bigg\|\sum_{g\in S} e^{p\e \beta(g)}\bigg\|_\infty\cleq \begin{cases}
(n+1)^{d}\: e^{p\e n} & \textrm{ if } p \e >\delta/2,\\[3pt]
(n+1)^{d+1}\: e^{p\e n} & \textrm{ if } p \e =\delta/2,\\[3pt]
e^{(\delta-p\e)n} & \textrm{ if } p \e <\delta/2.
\end{cases}
\end{align}
\end{lem}

\begin{proof} Assume first that $S\subseteq S_o(n)$. Fix $h\in \G$. 

Note that $\beta(g)(h)=d(o,ho)-d(go,ho)$ is an integer, as the distance on $X$ is integer-valued. For each $g\in S$, we have $|\beta(g)(h)|\leq d(o,go)=n$. Thus, if we consider the horospherical set
\[H(j)=\{g\in \G: \beta(g)(h)=j\},\]
then we can write
\begin{align*}
\sum_{g\in S} e^{p\e \beta(g)(h)}=\sum_{j=-n}^n  |S\cap H(j)|\:e^{p \e j}.
\end{align*}

Let $g\in S\cap H(j)$, where $j\in \{-n,\dots,n\}$. Let $m=m(g,h)\in X$ be a rough median point for the triple $o, go, ho$. For $x\in \{o, go, ho\}$, we introduce the shorthand $\delta(x)=d(m,x)$. See the sketch below.

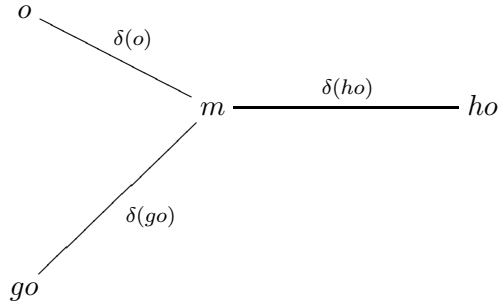
\begin{equation*}
\xymatrix{
o \ar@{-}[rrd]^{\delta(o)}&&&&&\\
&& m \ar@{-}[rrr]^{\delta(ho)} &&& ho\\
&&&&&\\
go \ar@{-}[rruu]_{\delta(go)} &&&&&\\}
\end{equation*}
\smallskip

We have
\begin{align*}
d(o,ho)\leq \delta(o)&+\delta(ho)\leq d(o,ho)+\rho,\\
d(go,ho)\leq \delta(go)&+\delta(ho)\leq d(go,ho)+\rho.
\end{align*}
By subtracting the two inequalities, we see that 
\[\delta(o)-\delta(go)\approx_\rho d(o,ho)-d(go,ho)=b(g)(h)=j,\]
where $\approx_\rho$ denotes the relation of being within $\rho$ of each other.  On the other hand
\[\delta(o)+\delta(go)\approx_\rho d(o,go)=n.\]
The two near equalities combine to give
\[\delta(o)\approx_\rho (n+j)/2, \qquad \delta(go)\approx_\rho (n-j)/2.\]

In view of the hypothesis \textsc{pol($d$)} on the graph $X$, the number of possibilities for $m$ is at most a constant multiple of $((n+j)/2+1)^d$. Next, by using the hypothesis \textsc{exp($\delta$)} on the action of $\G$ on $X$, we deduce that, for each such $m$, the number of possibilities for $g$ is at most a constant multiple of $e^{\delta(n-j)/2}$. Overall, we find that
\[|S\cap H(j)|\cleq \big((n+j)/2+1\big)^d\: e^{\delta(n-j)/2}\asymp (n+j+1)^d\: e^{\delta(n-j)/2} .\]
Therefore
\begin{align*}
\sum_{g\in S} e^{p \e \beta(g)(h)}\cleq& \sum_{j=-n}^n (n+j+1)^d\: e^{\delta(n-j)/2}\: e^{p \e j}\\
&=\sum_{j=0}^{2n}  (j+1)^d\: e^{\delta(2n-j)/2}\: e^{p \e (j-n)}\\
&=e^{(\delta-p \e)n}\sum_{j=0}^{2n} \:(j+1)^d\: e^{(p\e-\delta/2)j}.\end{align*}
Since the above estimate is uniform in $h\in \G$, we have
\begin{align}
\bigg\|\sum_{g\in S} e^{p\e \beta(g)}\bigg\|_\infty\cleq e^{(\delta-p \e)n}\sum_{j=0}^{2n} \:(j+1)^d\: e^{(p\e-\delta/2)j}.\label{eq: messy}
\end{align}
To estimate the latter sum, we use the following elementary asymptotics:
\begin{align}\label{eq: elem}
\sum_{j=0}^{N} \:(j+1)^{d}\: e^{\tau j} \asymp 
\begin{cases}
(N+1)^d\: e^{\tau N} & \textrm{ if } \tau >0\\[3pt]
(N+1)^{d+1} & \textrm{ if } \tau =0\\[3pt]
1 & \textrm{ if } \tau <0
\end{cases}
\end{align}
where the implicit multiplicative constants depend on $d$ and $\tau$, but not on $N$. We deduce that the right-hand side of \eqref{eq: messy} has the following asymptotics:
\begin{align*}
\begin{cases}
e^{(\delta-p\e)n}\: (2n+1)^{d}\: e^{(p\e-\delta/2)2n}\asymp (n+1)^{d}\: e^{p\e n} & \textrm{ if } p \e >\delta/2,\\[3pt]
e^{(\delta-p\e)n}\: (2n+1)^{d+1}\asymp (n+1)^{d+1}\: e^{p\e n} & \textrm{ if } p \e =\delta/2,\\[3pt]
e^{(\delta-p\e)n} & \textrm{ if } p \e <\delta/2.
\end{cases}
\end{align*}
To summarize: we have shown that, for a subset $S\subseteq S_o(n)$, we have
\begin{align}\label{eq: sph-cases}
\bigg\|\sum_{g\in S} e^{p\e \beta(g)}\bigg\|_\infty\cleq \begin{cases}
(n+1)^{d}\: e^{p\e n} & \textrm{ if } p \e >\delta/2,\\[3pt]
(n+1)^{d+1}\: e^{p\e n} & \textrm{ if } p \e =\delta/2,\\[3pt]
e^{(\delta-p\e)n} & \textrm{ if } p \e <\delta/2.
\end{cases}
\end{align}

We now address the general case when $S \subseteq B_o(n)$. By partitioning $S$ into spherical pieces, we can upper-bound
\begin{align}\label{eq: easy}
\bigg\|\sum_{g\in S} e^{p\e \beta(g)}\bigg\|_\infty\leq \sum_{k=0}^n \bigg\|\sum_{g\in S\cap S_o(k)} e^{p\e \beta(g)}\bigg\|_\infty.
\end{align}
Using \eqref{eq: sph-cases}, and then \eqref{eq: elem} once again, we see that the right-hand side of \eqref{eq: easy} can be asymptotically bounded from above by
\begin{align*}
 \begin{cases}
 \sum_{k=0}^n\: (k+1)^{d}\: e^{p\e k} \asymp (n+1)^{d}\: e^{p\e n} & \textrm{ if } p \e >\delta/2,\\[3pt]
 \sum_{k=0}^n\: (k+1)^{d+1}\: e^{p\e k}\asymp (n+1)^{d+1}\: e^{p\e n}& \textrm{ if } p \e =\delta/2,\\[3pt]
\sum_{k=0}^n\: e^{(\delta-p\e)k} \asymp e^{(\delta-p\e)n}& \textrm{ if } p \e <\delta/2.
\end{cases}
\end{align*}
All in all, we reach \eqref{eq: cases}.
\end{proof}

By combining Lemma~\ref{lem: work} with Lemma~\ref{lem: g-cocycle}, we derive the following asymptotic upper bounds. Note that the Busemann cocycle $\beta$ no longer appears in the statement!

\begin{thm}\label{lem: nearly}
Assume that the isometric action of $\G$ on $X$ satisfies the condition
\begin{itemize}[leftmargin=45pt]
\item[\textsc{exp($\delta$):}] there exist $\delta>0$ and $C>0$ such that, for each $x,y\in X$ and $n\in \N$, we have
$\big| \{g\in \G: d(x,gy)\leq n\}\big|\leq Ce^{\delta n}$.
\end{itemize}
Then, for any subset $S\subseteq B_o(n)$, we have
\begin{align}\label{eq: nobeta}
\begin{cases}
\big\|\lambda(\ct_S)\big\|_{p\to p}\cleq (n+1)^{d/p'} e^{\delta n/p} & \textrm{ if } p\in (1,2),\\[3pt]
\big\|\lambda(\ct_S)\big\|_{2\to 2}\cleq (n+1)^{d+1}\: e^{\delta n/2} & \textrm{ if } p=2.
\end{cases}
\end{align}
\end{thm}

\begin{proof}
For $p\in (1,2)$, we set $\e=\delta/(pp')$. Then $p \e=\delta/p'<\delta/2$ and $p' \e=\delta/p>\delta/2$. According to \eqref{eq: cases}, we have
\[\bigg\|\sum_{g\in S^{-1}} e^{p\e \beta(g)}\bigg\|_\infty \cleq e^{(\delta-p\e)n}, \qquad\bigg\|\sum_{g\in S} e^{p'\e \beta(g)}\bigg\|_\infty\cleq (n+1)^{d}\: e^{p'\e n}.\]
Hence, by \eqref{eq: concrete}
\begin{align*}
\big\|\lambda(\ct_S)\big\|_{p\to p} \cleq \big(e^{(\delta-p\e)n}\big)^{1/p}\cdot \big((n+1)^{d}\: e^{p'\e n}\big)^{1/p'}=(n+1)^{d/p'} e^{\delta n/p}.
\end{align*}

When $p=2$, we set $\e=\delta/4$. Then $p \e=\delta/2=p' \e$, and \eqref{eq: cases} yields
\[\bigg\|\sum_{g\in S^{-1}} e^{p\e \beta(g)}\bigg\|_\infty ,\; \bigg\|\sum_{g\in S} e^{p'\e \beta(g)}\bigg\|_\infty\cleq (n+1)^{d+1}\: e^{\delta n/2}.\]
Now \eqref{eq: concrete} says, quite simply, that
\begin{align*}
\big\|\lambda(\ct_S)\big\|_{2\to 2} \cleq (n+1)^{d+1}\: e^{\delta n/2}
\end{align*}
as well.
\end{proof}

The $p$-operator norm $\big\|\lambda(\ct_S)\big\|_{p\to p}$ is non-decreasing as a function of $S$. Therefore \eqref{eq: nobeta} is at its strongest when $S$ is the whole ball $B_o(n)$. But this is not the only case of interest--we will also apply Theorem~\ref{lem: nearly} to $S$ being the sphere $S_o(n)$, or even an annulus contained in $B_o(n)$.

\section{Geometric actions and pure growth}
We maintain the standing assumptions on $X$. But we now place a stronger requirement on the proper isometric action of $\G$ on $X$--namely, that the action be cocompact. A proper and cocompact isometric action is said to be a \emph{geometric action}. When the action of $\G$ on $X$ is geometric, the condition \textsc{exp($\delta$)} is equivalent to
\begin{itemize}[leftmargin=45pt]
\item[\textsc{e($\delta$):}] for some exponent $\delta>0$ we have $|B_o(n)|\cleq e^{\delta n}$.
\end{itemize}

Our main goal is that of obtaining asymptotic upper bounds for $\|\lambda_S\|_{p \to p}$, the $p$-norm of the averaging operator $\lambda_S=|S|^{-1}\cdot \lambda(\ct_S)$. We therefore recast \eqref{eq: nobeta}, which takes the form $\|\lambda(\ct_S)\|_{p \to p}\cleq f(n)e^{\delta n/p}$, into $\|\lambda_S\|_{p \to p}\cleq f(n) e^{\delta n/p}|S|^{-1}$. Now $e^{\delta n/p}|S|^{-1}\asymp |S|^{-1/p'}$ provided that $|S|\asymp e^{\delta n}$. In summary, we can deduce that $\|\lambda_S\|_{p \to p}\cleq f(n) |S|^{-1/p'}$ under the exponential growth assumption $|S|\asymp e^{\delta n}$.

This leads us to the next two results, our main applications of Theorem~\ref{lem: nearly}. 

\begin{thm}\label{thm: balls}
Assume that $\G$ acts geometrically on $X$, and that
\begin{itemize}[leftmargin=45pt]
\item[\textsc{b($\delta$):}] for some exponent $\delta>0$ we have $|B_o(n)|\asymp e^{\delta n}$.
\end{itemize}
Then
\begin{align}\label{eq: balls}
\begin{cases}
\big\|\lambda_{B_o(n)}\big\|_{p\to p}\cleq (n+1)^{d/p'}\: |B_o(n)|^{-1/p'} & \textrm{ if } p\in (1,2),\\[3pt]
\big\|\lambda_{B_o(n)}\big\|_{2\to 2}\cleq (n+1)^{d+1}\:  |B_o(n)|^{-1/2} & \textrm{ if } p=2.
\end{cases}
\end{align}
\end{thm}

\begin{thm}\label{thm: annuli}
Assume that $\G$ acts geometrically on $X$, and that
\begin{itemize}[leftmargin=45pt]
\item[\textsc{a($\delta$):}] for some exponent $\delta>0$ and thickness $\theta\geq 0$, the annuli $A_\theta(n)=\{g\in \G: n-\theta\leq d(o,go)\leq n+\theta\}$ satisfy $|A_\theta(n)|\asymp e^{\delta n}$.
\end{itemize}
Then
\begin{align}\label{eq: annuli}
\begin{cases}
\big\|\lambda_{A_\theta(n)}\big\|_{p\to p}\cleq (n+1)^{d/p'}\:|A_\theta(n)|^{-1/p'} & \textrm{ if } p\in (1,2),\\[3pt]
\big\|\lambda_{A_\theta(n)}\big\|_{2\to 2}\cleq (n+1)^{d+1}\: |A_\theta(n)|^{-1/2} & \textrm{ if } p=2.
\end{cases}
\end{align}
\end{thm}

We note that the \emph{annular pure growth} condition \textsc{a($\delta$)} implies the \emph{ball pure growth} condition \textsc{b($\delta$)} which, in turn, obviously implies the exponential growth condition \textsc{e($\delta$)}. Let us also note that, owing to the cocompactness of the action, each one of the conditions \textsc{a($\delta$)}, \textsc{b($\delta$)}, and \textsc{e($\delta$)}, is actually independent of the basepoint $o\in X$: if it holds for some basepoint, then it holds for any other basepoint.

In the next theorem, we highlight a significant instance of annular pure growth, namely the case of \emph{spherical pure growth}. This special case has an additional feature: it yields asymptotic upper bounds for $p$-operator norms of \emph{radial functions}, that is to say finitely supported functions $a: \G\to \C$ with the property that the value $a(g)$ only depends on $d(o,go)$.

\begin{thm}\label{thm: spheres}
Assume that $\G$ acts geometrically on $X$, and that
\begin{itemize}[leftmargin=45pt]
\item[\textsc{s($\delta$):}] for some exponent $\delta>0$ we have $|S_o(n)|\asymp e^{\delta n}$.
\end{itemize}
Then
\begin{align}\label{eq: spheres}
\begin{cases}
\big\|\lambda_{S_o(n)}\big\|_{p\to p}\cleq (n+1)^{d/p'}\: |S_o(n)|^{-1/p'} & \textrm{ if } p\in (1,2),\\[3pt]
\big\|\lambda_{S_o(n)}\big\|_{2\to 2}\cleq (n+1)^{d+1}\: |S_o(n)|^{-1/2} & \textrm{ if } p=2.
\end{cases}
\end{align}
Furthermore, if $a:\G\to \C$ is a radial function supported in the ball $B_o(n)$, then
 \begin{align}\label{eq: radial}
\begin{cases}
\|\lambda(a)\|_{p\to p}\cleq (n+1)^{(d+1)/p'} \|a\|_p & \textrm{ if } p\in (1,2),\\[3pt]
\|\lambda(a)\|_{2\to 2}\cleq (n+1)^{d+3/2}\: \|a\|_2 & \textrm{ if } p=2.
\end{cases}
\end{align}
\end{thm}

There is a small price to pay for the generality of the radial bounds \eqref{eq: radial}: compared to the spherical bounds \eqref{eq: spheres} and the ball bounds \eqref{eq: balls}, there is a slight degree loss in \eqref{eq: radial}.

\begin{proof}
The bounds \eqref{eq: spheres} are an instance of the bounds \eqref{eq: annuli}, in the case when $\theta=0$. We only need to argue the radial bounds \eqref{eq: radial}. We detail the case $p\in (1,2)$; the case $p=2$ is very similar, and we leave it to the reader.

Let $a:\G\to \C$ be a radial function supported in the ball $B_o(n)$. We write
\[a=\sum_{k=0}^n a_k \: \ct_{S_o(k)}.\]
By \eqref{eq: spheres}, we have $\big\|\lambda(\ct_{S_o(k)})\big\|_{p\to p}\cleq (k+1)^{d/p'} \big\|\ct_{S_o(k)}\big\|_p$ for each $k$. Therefore
\begin{align*}
\|\lambda(a)\|_{p\to p}\leq \sum_{k=0}^n |a_k| \: \big\|\lambda(\ct_{S_o(k)})\big\|_{p\to p}\cleq  \sum_{k=0}^n |a_k| \: (k+1)^{d/p'} \big\|\ct_{S_o(k)}\big\|_p. 
\end{align*}
Next, by means of H\"older's inequality, we can estimate
\begin{align*}
\sum_{k=0}^n |a_k| \: (k+1)^{d/p'} \big\|\ct_{S_o(k)}\big\|_p &\leq \Bigg(\sum_{k=0}^n (k+1)^{d}\Bigg)^{1/p'}  \Bigg(\sum_{k=0}^n |a_k|^p  \big\|\ct_{S_o(k)}\big\|^p_p\Bigg)^{1/p}.
\end{align*}
Since
\[\sum_{k=0}^n (k+1)^{d}\asymp (n+1)^{d+1},\] and
\[\sum_{k=0}^n |a_k|^p  \big\|\ct_{S_o(k)}\big\|^p_p=\sum_{k=0}^n |a_k|^p |S_o(k)|=\|a\|^p_p,\]
we obtain $\|\lambda(a)\|_{p\to p}\cleq (n+1)^{(d+1)/p'} \|a\|_p$, as claimed.
\end{proof}

In comparing the two upper bounds in \eqref{eq: spheres}, we might reiterate a point we made while discussing the Cohen--Pytlik asymptotic formulas \eqref{eq: cohen} and \eqref{eq: pytlik}: somewhat strikingly, the upper bounds in the range $p\in (1,2)$ are significantly better than those obtained by interpolation from the endpoint $p=2$. Indeed, using \eqref{eq: spheres} at $p=2$, and \eqref{eq: interpol}, we would get
\[\big\|\lambda_{S_o(n)}\big\|_{p\to p}\leq \big\|\lambda_{S_o(n)}\big\|_{2\to 2}^{2/p'}\cleq  (n+1)^{2(d+1)/p'}\: |S_o(n)|^{-1/p'}.\]
By comparison, \eqref{eq: spheres} at $p\in (1,2)$ has a polynomial factor of $(n+1)^{d/p'}$.

The main novelty in Theorems~\ref{thm: balls}, \ref{thm: annuli}, and \ref{thm: spheres} is the asymptotic upper bounds in the range $p\in (1,2)$. The bounds at $p=2$ are not exactly novel; they fall under the bigger scope of property RD (for Rapid Decay), which is known in the geometric situation exploited in this paper (cf. \cite{CR, Sap}). Our cocycle-based approach to these bounds is, however, novel. The other distinguishing feature has to do with the `degree of decay'. Most works dealing with property RD are content with \emph{some} polynomial factor $(n+1)^D$, whose degree $D$ often remains unspecified. Here we are very much interested in optimal degrees--an idea first put forth in \cite{Nic10}, and pursued further in \cite{Nic17}.

\section{Expansion}\label{sec: exp}
We begin with a general setup: $\G$ is a countable discrete group, and $S\subseteq \G$ is a fixed, non-empty, finite subset. For any non-empty finite subset $X\subseteq \G$, we consider the product set $SX=\{sx: s\in S,x\in X\}\subseteq \G$. We are interested in measuring the size of $SX$, relative to the size of $X$. Clearly, $|SX|\leq |S||X|$ and $|SX|\geq |X|$. 

\begin{defn} 
The \emph{expansion} of a finite subset $S\subseteq \G$ is defined as
\[e(S)=\inf \left\{\frac{|SX|}{|X|}: X \textrm{ non-empty finite subset of }\G \right\}.\]
\end{defn}

To put it differently, $e(S)$ is the largest constant with the property that $|SX|\geq e(S)|X|$ for each finite subset $X\subseteq \G$. The generic estimate for the expansion of $S$ is that
\[1\leq e(S)\leq |S|.\]

The notion of subset expansion, defined above, is a relative of well-known graph theoretical notions such as vertex expansion, and Cheeger constant. Despite its simple-minded formulation, we have not been able to locate it in the literature. Subset expansion, as defined above, is implicit in \cite[Sec.4]{Sap}; Sapir's considerations therein are the starting point for this section.  

Expansion is of topical interest in light of the following fact: $p$-operator norms for the averaging operator $\lambda_S$ serve as lower bounds for the expansion of $S$.

\begin{lem}
Let $p\in (1,\infty)$. Then
\begin{align}\label{eq: exp}
e(S)\geq \big\|\lambda_{S}\big\|_{p\to p}^{-p'}.
\end{align}
\end{lem}

\begin{proof}
Recall the duality pairing $\la \cdot, \cdot \ra: \ell^p\G\times \ell^{p'}\G\to \C$, defined by $\la \phi, \psi\ra=\sum_{h\in \G} \phi(h)\:\overline{\psi(h)}$. By H\"older's inequality, we have
\[\big|\la \lambda_{S}(\phi), \psi \ra\big|\leq \|\lambda_S(\phi)\|_p\cdot \|\psi\|_{p'}\leq \big\|\lambda_{S}\big\|_{p\to p}\: \|\phi\|_p \|\psi\|_{p'}. \]

Let $X\subseteq \G$ be a non-empty finite subset. By applying the above bound to $\phi=\ct_X$ and $\psi=\ct_{SX}$, we find that
\[\big|\la \lambda_S(\ct_X), \ct_{SX}\ra\big|\leq \big\|\lambda_{S}\big\|_{p\to p}\: |X|^{1/p}|SX|^{1/p'}.\]
Now 
\[ \lambda_S(\ct_X)=\frac{1}{|S|} \sum_{g\in S} \ct_{gX}\]
and so
\[ \la\lambda_S(\ct_X), \ct_{SX}\ra=\frac{1}{|S|} \sum_{g\in S}\la \ct_{gX},\ct_{SX}\ra=|X|\]
since, for each $g\in S$, we have $\la \ct_{gX},\ct_{SX}\ra=|gX\cap SX|=|gX|=|X|$. Thus
\[|X|\leq \big\|\lambda_{S}\big\|_{p\to p}\: |X|^{1/p}|SX|^{1/p'},\]
which can be rearranged as 
\[\big\|\lambda_{S}\big\|^{-p'}_{p\to p}\leq\frac{|SX|}{|X|}.\]
By taking the infimum over $X$, we obtain \eqref{eq: exp}.
\end{proof}

The interpolation bound \eqref{eq: interpol} says that
\[\big\|\lambda_{S}\big\|_{p\to p}^{-p'}\geq \big\|\lambda_{S}\big\|_{2\to 2}^{-2}\]
whenever $p\in (1,2)$. We therefore expect better lower bounds for $e(S)$ by using \eqref{eq: exp} in the range $p\in (1,2)$ rather than at the endpoint $p=2$.

By combining the above lemma with the results obtained in the previous section, we can deduce asymptotic lower bounds for the expansion of spheres, balls, or annuli. For the sake of simplicity, we only state the outcome in the case of spheres.

\begin{thm}\label{thm: expbelow}
Keep the notations and the assumptions of Theorem~\ref{thm: spheres}. Then
\[e(S_o(n))\succcurlyeq \frac{|S_o(n)|}{(n+1)^d}.\]
\end{thm}

In agreement with a point made above, it is critical to work with some $p$ in the range $(1,2)$ in order to deduce Theorem~\ref{thm: expbelow}. The bounds \eqref{eq: spheres} give
\[\big\|\lambda_{S_o(n)}\big\|_{p\to p}^{-p'}\succcurlyeq \frac{|S_o(n)|}{(n+1)^d}\]
for any $p\in (1,2)$, yet at $p=2$ we only have the much weaker bound
\[\big\|\lambda_{S_o(n)}\big\|_{2\to 2}^{-2}\succcurlyeq \frac{|S_o(n)|}{(n+1)^{2(d+1)}}.\]


\section{Hyperbolic groups}\label{sec: hyp}
The foremost example for us is the case when $\G$ is a non-elementary hyperbolic group. A choice of a finite, symmetric generating set for $\G$ defines a word length on $\G$, as well as a Cayley graph $X$ of $\G$, on which $\G$ acts geometrically. 

The first favorable fact is that the graph $X$ is roughly modular, and satisfies \textsc{pol($d$)} for $d=0$. We detail this fact in the following simple lemma.

\begin{lem}
Let $X$ be a connected, uniformly finite, hyperbolic graph. Then
\begin{itemize}
\item[(i)] for $\rho=2\delta+2$, where $\delta\geq 0$ is a hyperbolicity constant in the sense of Rips's thin triangle condition, we have $[x,y]_\rho\cap [y,z]_\rho \cap [z,x]_\rho\neq \emptyset$ for all $x,y,z\in X$;
\item[(ii)] for each $\rho\geq 0$ there exists $C'>0$ so that $| \{z\in [x,y]_\rho: d(x,z)=n\}|\leq C'$ for each $x,y\in X$ and $n\in \N$.
\end{itemize}
\end{lem}

\begin{proof}
We will use several times the following observation: if $\gamma_{xy}$ is a (discrete) geodesic joining $x$ to $y$, and $ \mathrm{dist}(m,\gamma_{xy})\leq c$, then $d(x,m)+d(m,y)\leq d(x,y)+2c$. 

(i) Let $x,y,z\in X$, and consider (discrete) geodesics $\gamma_{xy}$, $\gamma_{yz}$, $\gamma_{zx}$ joining $x$ to $y$, $y$ to $z$, respectively $z$ to $x$. Rips's thin triangle condition says that $\mathrm{dist}(v,\gamma_{yz})\leq \delta$ or $\mathrm{dist}(v,\gamma_{zx})\leq \delta$ for every vertex $v$ on $\gamma_{xy}$. It follows that there exists $v$ on $\gamma_{xy}$ such that $\mathrm{dist}(v,\gamma_{yz})\leq \delta+1$ and $\mathrm{dist}(v,\gamma_{zx})\leq \delta+1$. By the above observation, we have
\begin{align}\label{eq: rips}
\begin{cases}
d(y,v)+d(v,z)\leq d(y,z)+2(\delta+1),\\
d(z,v)+d(v,x)\leq d(z,x)+2(\delta+1).
\end{cases}
\end{align} 
In other words, $v\in [y,z]_{2\delta+2}$ and $v\in [z,x]_{2\delta+2}$; obviously, $v\in [x,y]_{2\delta+2}$ as well.

(ii) Let $\rho\geq 0$, and fix $x,y\in X$ as well as a (discrete) geodesic $\gamma_{xy}$ joining $x$ to $y$. Pick $z\in [x,y]_\rho$. Consider again a point $v$ on $\gamma_{xy}$ satisfying \eqref{eq: rips}. By adding the two inequalities we deduce that 
\[d(x,y)+2d(z,v)\leq d(x,z)+d(z,y)+4(\delta+1);\]
as $d(x,z)+d(z,y)\leq d(x,y)+\rho$, it follows that $d(z,v)\leq 2(\delta+1)+\rho/2=:\rho'$. 

Now, if $z\in [x,y]_\rho$ also satisfies $d(x,z)=n$, then $n-\rho'\leq d(x,v)\leq n+\rho'$. At most $2\rho'+1=\rho+4\delta+5$ vertices $v$ satisfy this bound since, we recall, $v$ lies on a geodesic $\gamma_{xy}$. For each such $v$ there are at most $N_\rho$ possibilities for $z$, where $N_\rho=\sup_{x\in X} |B_{\rho'}(x)|$; the finiteness of $N_\rho$ reflects the hypothesis that $X$ is uniformly finite. All in all, $| \{z\in [x,y]_\rho: d(x,z)=n\}|\leq (\rho+4\delta+5)N_\rho$, uniformly in $x,y\in X$ and $n\in \N$.
\end{proof}

The second favorable fact is that the Cayley graph $X$ enjoys pure spherical growth: $|S(n)|\asymp e^{\delta n}$ for some $\delta>0$. This is a result due to Coornaert \cite{Coo}. We are therefore in position to apply Theorem~\ref{thm: spheres}.

\begin{thm}\label{thm: hypspheres}
Let $\G$ be a non-elementary hyperbolic group, endowed with a word-length, and let $p\in (1,2)$. Then
\begin{align}\label{eq: hypspheres}
\big\|\lambda_{S(n)}\big\|_{p\to p}\cleq |S(n)|^{-1/p'}.
\end{align}
Furthermore, if $a:\G\to \C$ is a radial function supported in the ball $B(n)$, then
 \begin{align}\label{eq: hypradial}
\|\lambda(a)\|_{p\to p}\cleq (n+1)^{1/p'} \|a\|_p.
\end{align}
\end{thm}

The third favorable fact is that the asymptotic upper bound \eqref{eq: hypspheres} exactly complements the trivial lower bound \eqref{eq:triv-lower}, thereby implying Theorem~\ref{thm: main}.

As for expansion, Theorem~\ref{thm: expbelow} implies the following result--a restatement of which is Theorem~\ref{thm: exphyp}. 
\begin{thm}\label{cor: exphyp}
Let $\G$ be a non-elementary hyperbolic group, endowed with a word-length. Then
\[e(S(n))\succcurlyeq |S(n)|.\]
\end{thm}

Once again, this exactly complements the trivial upper bound $e(S(n))\leq |S(n)|$.

\section{Final remarks}

\begin{rem} 
Our motivating example in estimating $p$-operator norms for spherical averaging operators was the case of hyperbolic groups. The end result, Theorem~\ref{thm: spheres}, is quite a bit more general, and hyperbolic groups appear--in retrospect--as the simplest case. Pure exponential growth for balls, spheres, or annuli of radius $n$, of the form $e^{\delta n}$, can be generalized to mixed growth--that is, growth of the form $(n+1)^\alpha\: e^{\delta n}$. The techniques of this paper clearly go through. This should lead to estimates of $p$-operator norms for spherical averaging operators on cocompact lattices in $\mathrm{SL}_3(\R)$ and $\mathrm{SL}_3(\C)$ (cf. Lafforgue \cite{Laf}). 
\end{rem}

\begin{rem} Other contexts in which our results seem applicable include groups acting on finite dimensional CAT$(0)$ cube complexes, and groups that are hyperbolic relative to a family of subgroups of polynomial growth (cf. Chatterji--Ruane \cite{CR}). Under additional hypotheses, one also has pure annular growth for these families of groups by results of Yang \cite{Y}. 

Yet another context is that of groups which are cocompact lattices in products of hyperbolic graphs. The Cartesian product $X_1\square X_2$ of two graphs $X_1$ and $X_2$ is the graph with vertex set $X_1\times X_2$, and edges defined as follows: $(x_1,x_2)$ is adjacent to $(y_1,y_2)$ if $x_1=y_1$ and $x_2$ is adjacent to $y_2$, or if $x_1$ is adjacent to $y_1$ and $x_2=y_2$. The combinatorial distance in $X_1\square X_2$ is the sum metric $d\big((x_1,x_2), (y_1,y_2)\big)=d_{X_1}(x_1,y_1)+d_{X_2}(x_2,y_2)$. As a metric space, $X_1\square X_2$ is often referred to as the $\ell^1$-product of $X_1$ and $X_2$. The main point is that our standing assumptions are stable under taking Cartesian products. Specifically, we have the following fact: if $X_1$ and $X_2$ are roughly modular graphs, satisfying \textsc{pol($d_1$)} respectively \textsc{pol($d_2$)}, then their Cartesian product $X_1\square X_2$ is a roughly modular graph satisfying \textsc{pol($d_1+d_2+1$)}. In particular, the Cartesian product of $d-1$ uniformly finite, hyperbolic graphs is roughly modular and satisfies \textsc{pol($d$)}.
\end{rem} 

\begin{rem} Liao and Yu \cite[Defn.4.1]{LY} have introduced the following definition: a group $\G$ endowed with a length function has \emph{property RD$_p$}, where $1<p\leq 2$, if there exist constants $C>0$ and $D\geq 0$ such that
\[\|\lambda(a)\|_{p\to p}\leq C(n+1)^{D} \|a\|_p\]
for all functions $a:\G\to \C$ supported in the ball $B(n)$.

For $p=2$, this is the usual property RD; furthermore, it can be shown \cite[Thm.4.4]{LY} that property RD implies property RD$_p$ for all $p\in (1, 2]$. This has implications for the K-theory of Banach algebra completions of the group algebra $\C\G$ for varying exponents $p\in (1, 2]$.

From an analytic standpoint, it is interesting to find the optimal degree $D$ in the RD$_p$ bound, and to understand its dependence on $p$. Theorem~\ref{thm: spheres} gives RD$_p$-type bounds for constant spherical functions and, more generally, for radial functions, though the approach of this paper does not seem well-suited for the full RD$_p$ bound. Nonetheless, the radial RD$_p$ bounds suggest that the optimal degree $D$ in the range $p\in (1,2)$ is below what interpolation from the endpoint $p=2$ would predict.

Specifically, it seems likely that, for a non-elementary hyperbolic group $\G$ and $p\in (1,2)$, the optimal degree is $D=1/p'$; namely, the following should hold: there exists a constant $C>0$ such that
\[\|\lambda(a)\|_{p\to p}\leq C(n+1)^{1/p'} \|a\|_p\]
for all functions $a:\G\to \C$ supported in the ball $B(n)$. This is what we establish in \eqref{eq: hypradial} for radial functions. For comparison, we note that at the endpoint $p=2$ the optimal degree is $D=3/2$, and it is attained by radial functions \cite{Nic17}.
\end{rem}

\begin{rem} 
Asymptotic lower bounds for expansion, as in Theorem~\ref{thm: expbelow}, appear fairly naturally in connection with asymptotic upper bounds for $p$-operator norms of averaging operators. It would be interesting, however, to derive the bound of Theorem~\ref{thm: expbelow}--as well as improvements and variations thereof--in a direct, combinatorial way, without appealing to functional-analytic detours. The case of hyperbolic groups (Theorem~\ref{cor: exphyp}) seems particularly appealing in this respect.
\end{rem}

\begin{ack}
A very preliminary version of this paper was written during a research visit at Texas A\&M University in the Spring of 2017. I am grateful to Guoliang Yu for the invitation. I also thank Benben Liao for the fruitful conversations we had at the time around the preprint \cite{LY}.

The completion of this work was supported by a Ralph E. Powe enhancement award from the Oak Ridge Associate Universities.

Finally, I thank Kevin Boucher and J\'an \v{S}pakula for feedback.
\end{ack}


\begin{thebibliography}{oo} 

\bibitem{B} L. Bartholdi: \emph{Cactus trees and lower bounds on the spectral radius of vertex-transitive graphs}, in \emph{Random walks and geometry}, 349--361, Walter de Gruyter 2004

\bibitem{BC} L. Bartholdi, S. Cantat, T. Ceccherini-Silberstein, P. de la Harpe: \emph{Estimates for simple random walks on fundamental groups of surfaces}, Colloq. Math. 72 (1997), no.1, 173--193


\bibitem{CM} D.I. Cartwright, W. M\l{}otkowski: \emph{Harmonic analysis for groups acting on triangle buildings}, J. Austral. Math. Soc. Ser. A 56 (1994), no.3, 345-383


\bibitem{CR} I. Chatterji, K. Ruane: \emph{Some geometric groups with rapid decay}, Geom. Funct. Anal. 15 (2005), no.2, 311--339

\bibitem{Coh} J.M. Cohen: \emph{Operator norms on free groups}, Boll. Un. Mat. Ital. B (6) 1 (1982), no.3, 1055--1065

\bibitem{Coo} M. Coornaert: \emph{Mesures de Patterson--Sullivan sur le bord d'un espace hyperbolique au sens de Gromov}, Pacific J. Math. 159 (1993), 241--270

\bibitem{GT} E. Gardella, H. Thiel: \emph{Isomorphisms of algebras of convolution operators}, Ann. Sci. \'Ec. Norm. Sup\'er. (4) 55 (2022), no.5, 1433--1471

\bibitem{G} S. Gouezel: \emph{A numerical lower bound for the spectral radius of random walks on surface groups}, Combin. Probab. Comput. 24 (2015), no.6, 838--856


\bibitem{Kes} H. Kesten: \emph{Symmetric random walks on groups}, Trans. Amer. Math. Soc. 92 (1959), 336--354

\bibitem{Kes2} H. Kesten: \emph{Full Banach mean values on countable groups}, Math. Scand. 7 (1959), 146--156

\bibitem{Laf} V. Lafforgue: \emph{A proof of property (RD) for cocompact lattices of SL(3,$\R$) and SL(3,$\C$)}, J. Lie Theory 10 (2000), no.2, 255--267

\bibitem{LY} B. Liao, G. Yu: \emph{K-theory of group Banach algebras and Banach property RD}, preprint, arXiv:1708.01982

\bibitem{Nag} T. Nagnibeda: \emph{An upper bound for the spectral radius of a random walk on surface groups}, Zap. Nauchn. Sem. S.-Peterburg. Otdel. Mat. Inst. Steklov. (POMI) 240 (1997), 154--165, 293--294 (in Russian); English translation in J. Math. Sci. (NY) 96 (1999), no.5, 3542--3549


\bibitem{Nic10} B. Nica: \emph{On the degree of rapid decay}, Proc. Amer. Math. Soc. 138 (2010), no. 7, 2341--2347

\bibitem{Nic17} B. Nica: \emph{On operator norms for hyperbolic groups}, J. Topol. Anal. 9 (2017), no.2, 291--296

\bibitem{Phi} N.C. Phillips: \emph{Simplicity of reduced group Banach algebras}, preprint, arXiv:1909.11278

\bibitem{Pyt81} T. Pytlik: \emph{Radial functions on free groups and a decomposition of the regular representation into irreducible components}, J. Reine Angew. Math. 326 (1981), 124--135

\bibitem{Pyt84} T. Pytlik: \emph{Radial convolutors on free groups}, Studia Math. 78 (1984), no.2, 179--183

\bibitem{SW} E. Samei, M. Wiersma: \emph{Quasi-Hermitian locally compact groups are amenable}, Adv. Math. 359 (2020), 106897, 25 pp.

\bibitem{SW2} E. Samei, M. Wiersma: \emph{Exotic C$^*$-algebras of geometric groups}, J. Funct. Anal. 286 (2024), no.2, Paper No. 110228, 32 pp.

\bibitem{Sap} M. Sapir: \emph{The rapid decay property and centroids in groups}, J. Topol. Anal. 7 (2015), no.3, 513--541

\bibitem{Y} W.-Y. Yang: \emph{Statistically convex-cocompact actions of groups with contracting elements}, 
Int. Math. Res. Not. IMRN (2019), no.23, 7259--7323

\bibitem{Z} A. \.Zuk: \emph{A remark on the norm of a random walk on surface groups}, Colloq. Math. 72 (1997), no.1, 195--206

\end{thebibliography}
\end{document}